\documentclass{amsart}
\usepackage{amssymb, amsmath, amsfonts, mathrsfs}
\usepackage{todonotes}
\newtheorem{thm}{Theorem}[section]

\newtheorem{lem}[thm]{Lemma}
\newtheorem{prop}[thm]{Proposition}
\theoremstyle{definition}

\theoremstyle{remark}
\newtheorem{rem}[thm]{Remark}

\newtheorem{ex}[thm]{Example}
\numberwithin{equation}{section}

\begin{document}
\title{A note on infinite series with recursively defined terms}
\markright{Recursive series}
\author{Tam\'as Forg\'acs, Jack Luong and Joshua Williamson}

\maketitle

\begin{abstract}
In this note we study the convergence of recursively defined infinite series. We explore the role of the derivative of the defining function at the origin (if it exists), and develop a comparison test for such series which can be used even if the defining function of the series is not differentiable.
\end{abstract}

\section{Introduction}
In a recent paper in the {\it Monthly} (see \cite{efs}) the authors study the convergence properties of infinite series 
\begin{equation} \label{eq:generalseries}
\sum_{k=0}^{\infty} f^n(x_0), \qquad (f^n=\underbrace{f \circ f \circ \cdots \circ f}_{\textrm{$n$ times}}, \quad f^0=\operatorname{Id}),
\end{equation}
 for a certain class of real functions $f$. Their main result is the following
 \begin{prop}\label{prop:efsmain} (\cite[Proposition 1,\,p.360]{efs}) Consider the differentiable function $f:(0,\infty) \to (0,\infty)$ with the property that $f(x)<x$ for all $x \in (0,\infty)$, and the sequence with the properties:
\begin{itemize}
\item[(1)] $\lim_{n \to \infty}x_n=0$, with $x_n>0$ for all $n \in \mathbb{N}$;
\item[(2)] $x_{n+1}=f(x_n)$;
\item[(3)] the limit
\[
\lim_{x \to 0} \frac{x^a-f^a(x)}{x^af^a(x)}
\]
exists and equals $\dfrac{1}{k^a}$ for some $a,k \in (0,\infty)$.
\end{itemize}
Then:
\begin{itemize}
\item[(i)] the limit $\lim_{n \to \infty} n^{\frac{1}{a}}x_n$ exists and equals $k$;
\item[(ii)] the series $\sum_{n=1}^{\infty} x_n$ diverges if $a \geq 1$
\item[(iii)] the series $\sum_{n=  1}^{\infty} x_n$ converges if $a<1$.
\end{itemize}
 \end{prop}
 The purpose of this note is twofold. We first investigate the class of functions $\mathscr{F}$ to which Proposition \ref{prop:efsmain} applies, then we proceed to study the convergence properties of (\ref{eq:generalseries}) for $f \notin \mathscr{F}$. Our main result (Theorem \ref{thm:monotonemajorant}) extends the scope of Proposition \ref{prop:efsmain} to functions majorized by a monotone function. Along the way we also completely classify (partly using Proposition \ref{prop:efsmain}) those funtions which are differentiable at the origin and generate convergent series as in equation (\ref{eq:generalseries}).  The second goal is to investigate the ways in which the established results carry over to functions that are not necessarily positive, and whose domain is the real line. \\
 
 Before we get to the heart of the discussion, we note that the assumption of property (1) in the statement of Proposition \ref{prop:efsmain} is superfluous. Indeed, if $f:(0,\infty) \to (0,\infty)$ satisfies $f(x)<x$ for all $x \in (0,\infty)$, and $x_{n+1}=f(x_n)$, then the sequence $\left\{x_n \right\}_{n=0}^{\infty}$ is monotone decreasing and bounded below by zero. Were the $x_n$s to converge to $L>0$, the continuity of $f$ would imply $f(x_n) \to f(L)<L$, but of course $\lim f(x_n)=\lim x_n=L$. Thus property (1) follows from the assumptions on $f$ and property (2). In addition, we also see that extending $f$ to $[0,\infty)$ by setting $f(0)=0$ yields a continuous function. There is no way to guarantee however, that the extension is differentiable at $x=0$. In the interest of simplifying notation throughout the paper limits are not indicated as one sided even when they clearly are, and we also do not explicitly concern ourselves with the value $x_0$, beyond requiring that it seeds a sequence which converges to zero.
  \section{The limit in property (3)} In this section we show that if $f$ is in fact differentiable at zero, then $f'(0)=1$ must hold in order for the limit to exist (c.f. closing remarks in \cite{efs} discussing Applications 1, 2, 5 and 6). We then provide a proof of the fact that if $f$ is real analytic on a neighborhood of the origin with $f'(0)=1$, and $x_0$ is chosen in the domain of analyticity,  the resulting infinite series never converges. This result handles all but two of the Applications in \cite{efs} at once in terms of concluding that the series in question diverges. 
\newline Recall that $f$ as in the statement of Proposition \ref{prop:efsmain} always has a continuous extension to $[0,\infty)$, but not necessarily a differentiable one. 
 \begin{ex}\label{ex:nondiff} Let $f:(0,\infty) \to (0, \infty)$ defined by $\displaystyle{f(x)=x\left(\frac{1}{2}+\frac{1}{3}\sin\left(\frac{1}{x} \right) \right)}$ and extend it to the origin by $f(0)=0$. Then $f(x)<x$ for all $x>0$, but 
 \[
 \lim_{x \to 0} \frac{x\left(\frac{1}{2}+\frac{1}{3}\sin\left(\frac{1}{x} \right) \right)-0}{x-0}=\lim_{x \to 0} \left(\frac{1}{2}+\frac{1}{3}\sin\left(\frac{1}{x} \right) \right)=D.N.E,
 \]
 and hence the unique continuous extension to zero is non-differentiable there.
 \end{ex} 
 \begin{prop}\label{prop:f'(0)} Suppose that $I$ is an interval containing the origin, that $f:I \to \mathbb{R}$ satisfies $0<f(x)<x$ for $x>0$, and suppose that $f$ is differentiable at $x=0$. The limit
\begin{equation*} \label{eq:limit2}
\lim_{x \to 0} \frac{x^a - f(x)^a}{x^af(x)^a}
\end{equation*}
exists for some $a>0$ if and only if $f'(0) = 1$.
\end{prop}
\begin{proof}
We may write $f$ as
\begin{equation*}
f(x) =f'(0)x+\varepsilon(x,0),
\end{equation*}
where $\lim_{x \to 0} \frac{\varepsilon(x,0)}{x}=0$. It follows that
\begin{eqnarray*}
\lim_{x \to 0} \frac{x^a - f(x)^a}{x^af(x)^a} &=& \lim_{x \to 0} \frac{x^a - (f'(0)x+\varepsilon(x,0))^a}{x^a(f'(0)x+\varepsilon(x,0))^a}\\
&=& \lim_{x \to 0} \frac{x^a - x^a(f'(0)+\varepsilon(x,0)/x)^a}{x^{2a}(f'(0)+\varepsilon(x,0)/x)^a}\\
&=& \lim_{x \to 0} \frac{1-(f'(0)+\varepsilon(x,0)/x)^a}{x^a(f'(0)+\varepsilon(x,0)/x)^a}
\end{eqnarray*}
It is now clear that the limit exists if and only if $f'(0)=1$. More explicitly, since
\[
\lim_{x \to 0}x^a(f'(0)+\varepsilon(x,0)/x)^a=0,
\]
in order for the limit to exist, we must also have
\[
\lim_{x \to 0}1-(f'(0)+\varepsilon(x,0)/x)^a=0.
\]
This in turn implies that $f'(0)=1$.
 \end{proof}
 \begin{thm} Suppose that $f$ is real analytic on a neighborhood of the origin, and that it also satisfies the remaing assumptions of Proposition \ref{prop:efsmain}. Then the series generated by $f$ is divergent.
 \end{thm}
 \begin{proof}
Write 
\begin{equation*}
f(x) = \sum_{k=1}^{\infty} a_k x^k, \qquad a_1=1.
\end{equation*}
We now compute
\begin{eqnarray*}
\lim_{x \to 0} \frac{x^a - f(x)^a}{x^af(x)^a} &=& \lim_{x \to 0} \frac{x^a - (\sum_{k=1}^{\infty} a_kx^k)^a}{x^a(\sum_{k=1}^{\infty} a_kx^k)^a}
\\ &=& \lim_{x \to 0} \frac{x^a - x^a(1 + a_2x + a_3x^2 +\cdots)^a}{x^a (x +a_2 x^2+\cdots)^a}
\\ &=& \lim_{x \to 0} \frac{1 - (1+a_2x + \cdots)^a}{(x +a_2x^2+\cdots)^a}
\\ &=& \lim_{x \to 0} \frac{\mathcal{O}(x)}{x^a(1 +a_2x + ...)^a}
\end{eqnarray*}
which is non-zero and non-infinite only when the lowest power of $x$ in the numerator and that in the denominator are the same.  We conclude that $a \geq 1$ and hence the series generated by $f$ must diverge.	
\end{proof}
Having already imposed the restriction that $f$ be differentiable at zero, further requiring that $f'(0)=1$ limits the scope of Proposition \ref{prop:efsmain}. There are many functions which fail to satisfy this requirement, and yet would generate a convergent series.
\begin{ex} Suppose that $f$ is as in Proposition \ref{prop:f'(0)} with $f'(0)=1$, and that $\displaystyle{\sum  f^n(x_0)< +\infty}$. Set $h(x)=q f(x)$ for $q \in (0,1)$. Clearly $h'(0)=q<1$. Moreover, a simple induction argument shows that (using the notation of (\ref{eq:generalseries}))
\[
h^n(x_0)<q^n f(x_0), \qquad n \geq 1.
\]
Consequently, $\displaystyle{\sum  h^n(x_0) <+\infty}$, but the limit in (3) does not exists for $h$.
\end{ex}
Although there is no obvious way to modify the proposition so that it would apply to functions $f$ with $f'(0)<1$, it turns out that all such functions generate convergent series (as in \ref{eq:generalseries}).
\begin{thm} \label{thm:deriv<1}Suppose that $I$ is an interval containing the origin, that $f:I \to \mathbb{R}$ satisfies $0<f(x)<x$ for $x>0$, and suppose that $f$ is differentiable at $x=0$ with $f'(0)=c$ for some $0\leq c<1$. Then the series in (\ref{eq:generalseries}) is convergent.
\end{thm}
\begin{proof} Since $f'(0)=c$, for all $\varepsilon >0$ there exists a $\delta >0$ such that $|x| < \delta$ implies that $\displaystyle{\left|\dfrac{f(x)}{x}-c\right| <\varepsilon}$. We now select $\varepsilon >0$ small enough so that $\varepsilon + c<1$. The sequence $x_n \to 0$, so there is an $M \in \mathbb{N}$, such that $|x_n|<\delta$ for all $n \geq M$. It follows that for all such $n$, we have
\begin{eqnarray*}
-\varepsilon <& \dfrac{f(x_n)}{x_n}-c&< \varepsilon, \qquad \textrm{or equivalently} \\
x_n(\varepsilon-c) <& f(x_n)&<x_n(\varepsilon+c).
\end{eqnarray*}
We now calculate
\begin{eqnarray*}
\sum_{k=0}^{\infty} f^k(x_0)&=&\sum_{k=0}^{M-1} f^k(x_0)+\sum_{k=M}^{\infty} f^k(x_0)\\
&<&\sum_{k=0}^{M-1} f^k(x_0)+\sum_{k=0}^{\infty} f^k(x_M)\\
&<&\sum_{k=0}^{M-1} f^k(x_0)+\sum_{k=0}^{\infty} (\varepsilon+c)^k \cdot x_M\\
&<& \sum_{k=0}^{M-1} f^k(x_0)+\dfrac{x_M}{1-(\varepsilon+c)}<+\infty,
\end{eqnarray*}
where the inequality $f^k(x_M)<(\varepsilon+c)^k x_M$, $k=0,1,2,\ldots$ can be verified quickly by induction. The proof is complete.
\end{proof}

We thus dealt with all interesting cases when $f'(0)$ exists. For if $f'(0)<0$, the definition of the derivative would yield $f(x)<0$ for $x \ll1$. Similarly, if $f'(0)>1$, then $f(x)>x$ for $x \ll 1$. Both of these cases are ruled out by our hypotheses.
\section{A comparison test} It is natural to wonder whether the following is true: if $f$ generates a convergent series, and $g \leq f$ on $(0,\infty)$, then $g$ generates a convergent series. The difficulty with demonstrating such a fact lies in the dynamics of the problem. Clearly, $g(x_0) \leq f(x_0)$, but it does not follow that $g(g(x_0)) \leq f(f(x_0))$. Thus term-wise comparison of the two generated series is not possible without additional information. However, if one of the two functions in question is (eventually) monotonic, the argument becomes easier.
\begin{thm} \label{thm:monotonemajorant} Suppose that $f:(0,\infty) \to (0, \infty)$ is monotone increasing, and that $\sum_{n=1}^{\infty} f^n(x_0)$ is convergent. If $0<g(x)\leq f(x)<x$ for $x>0$, then $\sum_{n=1}^{\infty} g^n(x_0)$ converges as well.
\end{thm}
\begin{proof} Let $x_0 \in \mathbb{R}$, let $f$ be monotone increasing and suppose that $0<g(x)\leq f(x)<x$ for $x>0$. It is then clear that $g(x_0) \leq f(x_0)$. Suppose now that $g^k(x_0) <f^k(x_0)$ for some $k \geq 1$.  Then 
\[
f^{k+1}(x_0)=f(f^k(x_0)) \stackrel{(\star)}{\geq} f( g^k(x_0)) \stackrel{(\star \star)}{\geq} g(g^k(x_0))=g^{k+1}(x_0),
\]
where $(\star)$ is a consequence of the monotonicity of $f$, and $(\star \star)$ follows because $g(x) \leq f(x)$ for all $x>0$. By induction we see that $f^n(x_0) \geq g^n(x_0)$ for all $n \in \mathbb{N}$. Whence, by the comparison test $\sum g^n(x_0)$ converges, since $\sum f^n(x_0)$ does.
\end{proof}
\begin{rem} The assumptions in the above theorem can be relaxed to requiring only that $f$ be monotone increasing on the interval $(0, \delta)$ for some $\delta >0$. For if $f$ is monotone on $(0,\delta)$, then there is an $n \in \mathbb{N}$ such that $g^n(x_0) \in (0,\delta)$ for all $n \geq N$. Setting $y_0=g^N(x_0)$ returns us to the setting of the theorem, by which the convergence of $\sum f^n(y_0)$ implies the convergence of $\sum g^n(y_0)$, and hence that of $\sum g^n(x_0)$. 
\end{rem}
With Theorem \ref{thm:monotonemajorant} in hand we now revisit condition (3) in Proposition \ref{prop:efsmain}. If
\[
\lim_{x \to 0} \frac{x^a-f^a(x)}{x^af^a(x)}=\dfrac{1}{k^a}
\]
for some $a,k \in (0,\infty)$, then for $x \ll 1$ the inequality 
\[
\frac{x^a-f^a(x)}{x^af^a(x)} \geq c >0 
\]
holds for some $c>0$. Consequently,
\[
x >\dfrac{x}{\sqrt[a]{1+cx^a}} \geq f(x)  \qquad (x \ll1).
\]
Note that 
\[
\dfrac{d}{dx} \left(\dfrac{x}{\sqrt[a]{1+cx^a}} \right)=(1+c x^a)^{-\frac{1+a}{a}}>0,  \qquad (x >0)
\]
and consequently this function is monotone increasing on $(0,\infty)$. Condition (3) thus implies that $f$ is majorized by a monotone function on an interval $(0,\delta)$, for some $\delta >0$. Since Proposition \ref{prop:efsmain} tells us that in case the limit in (3) exists with $a \geq 1$, the series diverges, it is only interesting to investigate the function $x/\sqrt[a]{1+c x^a}$ for $0<a<1$.
\begin{prop} Suppose $0<a<1$, $x_0>0$, and let 
\[
g(x)=\dfrac{x}{\sqrt[a]{1+c x^a}}
\]
for some $c>0$. Then the series $\sum_{n=1}^{\infty} g^n(x_0)$ converges.
\end{prop}
\begin{rem} This result is established in \cite[p.\,362]{efs} for $c=1$, and $a=\frac{1}{n}, \ n \geq 2$ in Application 3. 
\end{rem}
\begin{proof} Note that for $0<a<1$
\[
\lim_{x \to 0} \dfrac{x^a-\left(\frac{x}{\sqrt[a]{1+c x^a}} \right)^a}{x^a\left(\frac{x}{\sqrt[a]{1+c x^a}} \right)^a}=c>0.
\]
An application of Proposition \ref{prop:efsmain} completes the proof.
\end{proof}
We close this section with a final remark. The assumptions of Theorem \ref{thm:deriv<1} imply that $0<f(x)<c_1x$ on some $(0, \delta)$ and $0<c_1<1$, and hence the recursive series generated by $f$ converges, since the one generated by $g(x)=c_1x$ does. Theorem \ref{thm:monotonemajorant} achieves the same conclusion under much weaker hypotheses. In particular, while Theorem \ref{thm:deriv<1} does not apply to the function $\displaystyle{f(x)=x\left(\frac{1}{2}+\frac{1}{3}\sin\left(\frac{1}{x} \right) \right)}$ (c.f. Example \ref{ex:nondiff}), Theorem \ref{thm:monotonemajorant} allows us to conclude the series generated by this function is in fact convergent, since 
\[
x\left(\frac{1}{2}+\frac{1}{3}\sin\left(\frac{1}{x} \right) \right) \leq \frac{5}{6}x, \qquad (x>0)
\]
and the right hand side generates a convergent series.
\section{Series whose terms are not all of the same sign}
We now widen our scope of investigation to functions $f:\mathbb{R} \to \mathbb{R}$ which satisfy $|f(x)|< |x|$ for all $x \neq 0$ and $f(0)=0$. If the graph of $f$ lies in the first quadrant for $x>0$ (regardless of what it looks like for $x<0$), or if the graph lies in the third quadrant for $x<0$ (regardless of what it looks like for $x>0$), then the analysis of the generated series reduces to the cases we discussed in the previous sections of this note. The next lemma shows that if the graph of $f$ is contained in the second and fourth quadrants, then $f$ will always generate a convergent series.
\begin{lem}
Suppose $f: \mathbb{R} \to \mathbb{R} $ satisfies $|f(x)|<|x|$ for $x \neq 0$ and $f(0)=0$. Suppose in addition that $x f(x)<0 $ for all $x \in \mathbb{R} \setminus \{0\}$. The series defined in (\ref{eq:generalseries}) converges.
\end{lem}
\begin{proof}
The condition $x f(x)<0$ implies that the terms of the sequence $\{x_n\}$ alternate in sign. Since $|x_n| \to 0$, the alternating series test gives the result.
\end{proof}
Thus it remains to investigate functions $f$ whose graph - over any interval containing zero - intersects all four quadrants. We note that if $f$ is such a function, and $|f(x)|<c|x|$ for some $c<1$ over any interval $(-\epsilon, \epsilon)$, then the series generated by $f$ is absolutely convergent.  \\
If the graph of $f$ over $(-\epsilon, \epsilon)$ intersects all four quadrants for every $\epsilon>0$, then $f$ has infinitely many zeros. Moreover, the sequence of zeros of $f$ have $0$ as a limit point. Consequently either $f$ is differentiable at $x=0$ with $f'(0)=0$, or $f$ is non-differentiable at the origin. In the first case Theorem \ref{thm:deriv<1} \textit{mutatis mutandis} implies that the series $\sum f^n(x_0)$ is absolutely convergent. There is however very little we can say in case $f$ is non-differentiable at $x=0$ and $|f|$ is not majorized by $c|x|$ on any interval $(-\epsilon, \epsilon)$, and any $0 \leq c<1$. By the conditions imposed on $f$ we will always have $\{|x_n| \}_{n=0}^{\infty}$ a decreasing sequence of positive numbers with $\lim |x_n|=0$. It follows that the series $\sum f^n(x_0)=\sum \operatorname{sgn}(f^n(x_0)) |f^n(x_0)|$ will converge if the partial sums of $\sum \operatorname{sgn}(f^n(x_0))$ are bounded (see \cite[Theorem 3.42]{rudin}). Unfortunately we have no easy way to determine when this happens, and even if we did the characterization would not be complete. Thus, the problem remains open, but we propose the following problem for further contemplation by the reader:\\
\textbf{Problem:} Find all functions $g: \mathbb{N} \to \{ 0,1\}$ such that $\sum (-1)^{g(n)}$ has bounded partial sums.




\begin{thebibliography}{1}

\bibitem{efs} Evelyn~R.~Easdale, Jolene.~E.~Fleming, and Bogdan.~D.~Suceav\u a, Convergence for Series With Terms Defined by a Recurrence Relation, \textit{Amer. Math. Monthly,} \textbf{124} (2017) 360-364.

\bibitem{Monthly} Jolene~E.~Harris, Bogdan D.~Suceav\u{a}, Problem 11244, \textit{ Amer. Math. Monthly,} \textbf{113}, 759-760.

\bibitem{rudin} Walter Rudin, \textit{Principles of mathematical analysis}, McGraw-Hill, 1976.


\end{thebibliography}
\end{document}